\documentclass[10pt,]{article}
\usepackage{graphicx,amsthm}
\usepackage{amsfonts,amssymb,amscd,amsmath,enumerate,verbatim,latexsym,epsfig,amsfonts}

\newtheorem{Theorem}{Theorem}[section]
\newtheorem{Lemma}[Theorem]{Lemma}
\newtheorem{Corollary}[Theorem]{Corollary}

\newtheorem{Definition}[Theorem]{Definition}

\begin{document}
\title{On Resolvability of a Graph Associated to a Finite Vector Space}
\author{U. Ali\thanks{Corresponding author.}, S. A. Bokhary, K. Wahid
\vspace{4mm}\\
\normalsize  Centre for Advanced Studies in Pure and Applied Mathematics,\\\normalsize
Bahauddin Zakariya University, Multan, Pakistan\\\normalsize
{\tt uali@bzu.edu.pk, sihtsham@gmail.com, kholawahid12@gmail.com}}

\date{}
\maketitle


\begin{abstract}
 The metric dimension of non-component graph, associated to a finite vector space, is determined.
   It is proved that the exchange property holds for resolving sets of the graph, except a special case.
    Some results are also related to an intersection graph.
\end{abstract}

\textit{AMS Subject Classification Number}: 05C12, 05C25, 05C62\\

\textit{Keywords: Graph, vector space, metric dimension, resolving set, exchange property}\\

\section{Introduction}
Let $\mathbb{V}$ be a vector space over a field $\mathbb{F}$ with a basis $\{\alpha _{1},
\alpha_{2},\ldots,\alpha_{n}\}$. A vector $v\in V$ is expressed, uniquely, as a
linear combination of the form $v =a_{1}\alpha_{1}+a_{2}\alpha_{2}+\cdots+a_{n}\alpha_{n}$. A graph called \emph{Non-Zero
Component graph} $\Gamma(\mathbb{V})$ is associated, by Angsuman Das in~\cite{Das}, to a finite dimensional vector space in the following way:\\
the vertex set of the graph $\Gamma(\mathbb{V})$ is the non-zero vectors and two vertices are joined by an edge if they share at least one
$\alpha_{i}$ with non-zero coefficient in their unique linear combination with respect to $\{\alpha _{1},
\alpha_{2},\ldots,\alpha_{n}\}$. It is proved in~\cite{Das} that $\Gamma(\mathbb{V})$ is independent of choice of basis, i.e., isomorphic non-zero component graphs are obtained for two different bases. Other basic properties and results about $\Gamma(\mathbb{V})$ can be found in~\cite{Das, Das1}, for example
\begin{Theorem}[\cite{Das1}, Theoerem 3.6]\label{order}
  If $\mathbb{V}$ be a $n$-dimensional vector space over a finite field $\mathbb{F}$ with $q$ elements, then the order of $\Gamma(\mathbb{V})$ is $q^n-1$ and the size of $\Gamma(\mathbb{V})$ is
  $$\frac{q^{2n}-q^n+1-(2q-1)^n}{2}.$$
\end{Theorem}
\begin{Theorem} [\cite{Das}, Theoerem 4.2] \label{complet}
  $\Gamma(\mathbb{V})$ is complete if and only if  $\mathbb{V}$ is a $1$-dimensional vector space.
\end{Theorem}
The study of graphs associated to algebraic structures has, recently, got much attention from researchers to expose the relation between algebra and graph theory: zero divisor graph associated to a commutative ring was discussed in \cite{Anderson, Beck}; commuting graphs for groups were studied in \cite{bates1,bates2,com}; power graphs for groups and semigroups were focused in \cite{Cameron, Chakrabarty, Moghaddamfar}; the papers \cite{Rad,Talebi} are devoted to investigate intersection graphs assigned to a vector space; and so on.

In this paper, we study the metric dimension and the exchange property for resolving sets of $\Gamma(\mathbb{V})$ of a
  vector space $\mathbb{V}$. The vector space under consideration is of dimension $n$ and defined over a finite field $\mathbb{F}$ having $q$ elements.

\section{Metric Dimension}
Let $ \Gamma= (V, E)$ be a finite simple connected
graph. The \emph{distance} $d(u,v)$ between two vertices
$u,v\in V$ is the length of a shortest path between them.
Let $W=\{w_{1},w_{2},\dots ,$ $w_{k}\}$ be an ordered set of
vertices of $\Gamma$ and let $v$ be a vertex of $\Gamma$. The
\emph{representation} $r(v|W)$ of $v$ with respect to $W$ is the
$k$-tuple $(d(v,w_{1}),d(v,w_{2}),\ldots
,d(v,w_{k}))$. $W$ is called a \textit{resolving set}
\cite{Chartrand3} or \emph{locating set} \cite{Slater1} if every
vertex of ${\Gamma}$ is uniquely identified by its distances from
the vertices of $W$, or equivalently, if distinct vertices of
${\Gamma}$ have distinct representations with respect to $W$. Two vertices $v$
and $u$ are said to be resolved by a vertex $w\in W$ if $d(v,w)\neq d(u,w)$.
A resolving set of minimum cardinality is called a \textit{metric basis} for
${\Gamma}$ and this cardinality is the \emph{metric dimension}
of $\Gamma$ (see \cite{Buczkowski, Harary,   Imran2, tomescu,
jav}). A resolving set $W$ of cardinality $k$ is called  \emph{minimal resolving set}
if it does not contain a resolving set of cardinality $k-1$ as a proper subset. For other basic terms and concepts, the reader is referred to~\cite{West}.
An application of the metric dimension and basis of a graph in robot navigation problem is considered in~\cite{khuller} by S.
Khuller.

Given a vertex $u$ in a graph $\Gamma$, \emph{the open neighborhood}
of $u$ in $\Gamma$, denoted by $N(u)$, is the set
$$\{v\in V| d(u,v)=1\}.$$
And \emph{the closed neighborhood} of $u$ in $\Gamma$, denoted by
$N[u]$, is the set
$$\{v\in V | d(u,v)=1\}\cup\{u\}.$$
 For two vertices $u$ and $v$ in a graph $\Gamma$, define $u\equiv v$ if $N[u]=N[v]$ or $N(u)=N(v)$. The relation $\equiv$ is an
 equivalence relation (see~\cite{Hernando}). If $u\equiv v$ then $u$ and $v$ are called \emph{twins}.
\begin{Lemma}[\cite{Hernando}, Corollary 2.4. ]\label{verclass}
Suppose $u,v$ are twins in a connected graph $\Gamma$ and $W$ resolves $\Gamma$. Then $u$ or $v$ is in $W$.
Moreover, if $u\in W$ and $v\notin W$, then $(W\setminus \{u\})\cup \{v\}$ also resolves $\Gamma$.
\end{Lemma}
\begin{proof}
For any two vertices $u\equiv v$ a connected graph $\Gamma$. We have
$d(u,w)$=$d(v,w)$ for all $w \in
V\setminus \{u,v\}$. Therefore, no vertex can resolve $u$ and $v$, which forces either $u\in W$ or $v\in W$. Moreover, if two vertices are resolved by $u$ then they are also resolved by $v$. Hence, $(W\setminus \{u\})\cup \{v\}$ also resolves $\Gamma$.
\end{proof}

In order avoid confusion, we denote a basis
$\{\alpha_{1},\alpha_{2},\ldots,\alpha_{n}\}$ of a vector space $\mathbb{V}$ by $\mathbb{V}$-basis and a
minimum resolving set by  $\Gamma(\mathbb{V})$-basis. We denote the dimension of $\mathbb{V}$ by $dim_{\mathbb{V}}$ and
the metric dimension of $\Gamma(\mathbb{V})$ by $dim_{\Gamma(\mathbb{V})}$. The elements of the set $V$ represent non-zero vectors of the vector space
$\mathbb{V}$ and vertices of the graph $\Gamma(\mathbb{V})$ at the same time.
\begin{Definition} [\cite{Das1}]
 Skeleton $S_{v}$ of a vector $v$ is the set of $\alpha$'s with non-zero coefficients in the basic representation of $v$ with respect to $\mathbb{V}$-basis $\{\alpha_1,\alpha_2,\ldots, \alpha_n\}.$
\end{Definition}
 A vector $v$ is not unique for a skeleton $S_v$, i.e., two distinct vectors may have the same skeleton. The length of a skeleton $S_{v}$ is the number of vectors of $\mathbb{V}$-basis it contains.
For two vertices $u$ and $v$ in $\Gamma(\mathbb{V})$, we define the equivalence relation $u\cong v$ if $S_{u}=S_{v}$. Let $C_{u}$ denote the equivalence class containing $u$.\\
The following lemma plays a key role in the proof of our main theorems.
\begin{Lemma}\label{equi}
In $\Gamma(\mathbb{V})$, the two equivalence relations $\equiv$ and $\cong$ are equivalent, i.e., both relations bring the same partition to the set of vertices of $\Gamma(\mathbb{V})$ .
\end{Lemma}
\begin{proof}
Let $u\equiv v$. Then, $u$ and $v$ have the same neighborhood and we need to show that $S_{u}=S_{v}.$ Suppose on contrary that $u$ and $v$ have different skeletons. Without loss generality, there exists some $\alpha_i\in S_{u}$ and $\alpha_i\notin S_{v}$. Consequently, the vertex $\alpha_i$ is adjacent to $u$ and not adjacent to $v$ in $\Gamma(\mathbb{V})$. It contradicts that $u$ and $v$ have the same neighborhood. Conversely, let $u\cong v$ then, by definition of the graph $\Gamma(\mathbb{V})$, $u\equiv v$.
\end{proof}
\begin{Theorem}\label{Mertic Dimension}
Let $\mathbb{V}$ be a vector space over a field $\mathbb{F}$ of $q$ elements
 with $\{\alpha_{1},\alpha_{2},\ldots,\alpha_{n}\}$ as $\mathbb{V}$-basis:\\
a) if $q=2$ and $n=2$ then
$dim_{\Gamma(\mathbb{V})}= 1$;\\
b) if $q=2$ and $n\geq 3$ then
$dim_{\Gamma(\mathbb{V})}= n$; and\\
c) if $q\geq 3$ then
$dim_{\Gamma(\mathbb{V})}= \sum_{k=1}^n\binom{n}{k}((q-1)^k-1)$.
\end{Theorem}
\begin{proof}
a) The two vertices $\alpha_2$ and $\alpha_1+\alpha_2$ can be resolved by the vertex $\alpha_1$, hence $\Gamma(\mathbb{V})$-basis$=\{\alpha_1\}.$\\
 b) For $q=2$, a class $C_v$ for all $v\in V$ is of length $1$, i.e., $\mid C_v\mid =1$ and, by Lemma~\ref{equi}, all vertices $\Gamma(\mathbb{V})$ are singleton twins.
 Any two vertices, being singleton twins, can be resolved by some $\alpha_i\in\mathbb{V}$-basis. Therefore, the $\mathbb{V}$-basis is a resolving set. To show that this $\mathbb{V}$-basis is a $\Gamma(\mathbb{V})$-basis, we need to show that $dim_{\Gamma(\mathbb{V})}\geq n.$
 Let $u$ be the vertex whose skeleton is $S_u=\{\alpha_1,\alpha_2,\ldots,\alpha_n\}$, i.e., $u=\alpha_1+\alpha_2+\cdots+\alpha_n$ and $X=\{v_1,v_2,\ldots,v_n\}$ with the property that $S_{v_i}=\mathbb{V}$-basis$\setminus \{\alpha_i\}$ for $1\leq i\leq n$.\\
 \textit{Case1}. When $u\notin\Gamma(\mathbb{V})$-basis. The vertex $u$ and a vertex $v_i\in X$ can only be resolved by $\alpha_i$, because $d(u,v)=d(v_i,v)$ for all $v\neq\alpha_i$. Therefore, at least, one vertex from each pair $(\alpha_1, v_i)$ must be part of the $\Gamma(\mathbb{V})$-basis. There are $n$ such pairs, hence $dim_{\Gamma(\mathbb{V})}\geq n.$\\
 \textit{Case2}. When $u\in\Gamma(\mathbb{V})$-basis. Any two vertices $v_j,v_k\in X$ can only be resolved either by $\alpha_j$ or by $\alpha_k$, because $d(v_j,v)=d(v_k,v)$ for $v\neq\alpha_j,\alpha_k$. Therefore, at least, one vertex from each pair of $n-1$ pairs, of type $(\alpha_1, v_i)$, must be part of the $\Gamma(\mathbb{V})$-basis along with the vertex $u$ and hence the result follows.\\
 c) Since $q\geq3$, so $|C_v|\geq2$ for all $v\in V$ and, by Lemma~\ref{equi}, there is no singleton twin. By Lemma~\ref{verclass}, a resolving set must contain all vertices, except one, in every class $C_v$ of length $k$ for $1\leq k\leq n$; there are $(q-1)^k$ vertices in a class $C_v$ of length $k$ (there are $q-1$ choices and $k$ places); and for a fixed $k$, there are $\binom{n}{k}$ distinct classes. Therefore, any resolving set $W$ contains, at least, $\sum_{k=1}^n\binom{n}{k}((q-1)^k-1)$ vertices. There are $n$ distinct classes of length $1$ and, by Lemma~\ref{verclass}, $W$ contains the set $D=\{a_1\alpha_1,a_2\alpha_2,\ldots,a_n\alpha_n\}$, where $a_i\in \mathbb{F}$. Any two vertices having different skeletons can be resolved by some vertex in the set $D$, hence the minimum cardinality of $W$ is $\sum_{k=1}^n\binom{n}{k}((q-1)^k-1)$.
\end{proof}
\begin{Corollary}
For $q\geq 3$, any $\Gamma(\mathbb{V})$-basis contains a $\mathbb{V}$-basis.
\end{Corollary}
\begin{proof}
The set $D=\{a_1\alpha_1,a_2\alpha_2,\ldots,a_n\alpha_n\}$ in the proof part $c)$ of Theorem~\ref{Mertic Dimension} is a $\mathbb{V}$-basis.
\end{proof}
For $q=2,$ the above corollary is not true. For example, let $n=3$ the set
$ \{\alpha_1, \alpha_1+\alpha_3,\alpha_3\}$
is a $\Gamma(\mathbb{V})$-basis, but not a $\mathbb{V}$-basis.
\section{Exchange Property}
Resolving sets are said to have \emph{the exchange property} in a graph
$\Gamma$ if whenever $W_1$ and $W_2$ are minimal resolving sets for
$\Gamma$ and $r\in W_1$, then there exists $s\in W_2$ so that
$(W_2\setminus\{s\})\cup\{r\}$ is also a minimal resolving set
(see~\cite{Boutin}). If a graph $\Gamma$ has the exchange property,
then every minimal resolving set for $\Gamma$ has the same size and
algorithmic methods for finding the metric dimension of $\Gamma$ are
more feasible. Similarly exchange property for dominating sets of a
graph was also introduced in~\cite{Boutin}. In this paper, the
exchange property of a graph $\Gamma$ always means the property for
minimal resolving sets.

For a vector space $\mathbb{V}$, the exchange
property holds for $\mathbb{V}$-bases. Resolving sets behave
like $\mathbb{V}$-bases in a vector space $\mathbb{V}$, i.e., each vertex in the graph can be
uniquely identified relative to the vertices of a resolving set.
But, unlike $\mathbb{V}$-bases in vector spaces, resolving sets do not always
have the exchange property. To show the exchange property does not
hold in a graph, it is sufficient to show that there exist two minimal
resolving sets of different size. However, the condition is not
necessary, i.e., the exchange property does not hold
does not imply that there are minimal resolving sets of different
size.
Results about the exchange property for
different graphs can be found in the literature, for example
\begin{Theorem} [\cite{Boutin}, Theorem 3]
The exchange property holds for resolving sets in trees.
\end{Theorem}
\begin{Theorem} [\cite{Boutin}, Theorem 7]
For $n\geq 8$, the exchange property does not hold in
wheels $W_{n}$.
\end{Theorem}

Another sufficient condition, for a graph not to have the exchange property, can be  stated as
\begin{Lemma}\label{sufcon}
If there is a minimal resolving set of size $>dim_{\Gamma}$ in a graph $\Gamma$. Then, the exchange property does not hold in $\Gamma$.
\end{Lemma}
\begin{proof}
Since a minimum resolving set is also minimal so there are two minimal resolving sets of different size (one of size $dim_{\Gamma}$ and other is of size $>dim_{\Gamma}$).
\end{proof}
\begin{Lemma}\label{mresolving}
For $q=2$ and $n\geq 3$, the set $V_{n-1}=\{v\in V: \alpha_{n-1}\notin S_v\}$ is a minimal resolving set in $\Gamma(\mathbb{V})$.
\end{Lemma}
\begin{proof}
Any vertex $u\in V\setminus V_{n-1}$ is either $\alpha_n$ or of the form $w+\alpha_n$, where $w\in V_{n-1}.$ Let $u_1,u_2\in V\setminus V_{n-1}$. Then, there exist some $\alpha_i$ for $1\leq i\leq n-1$ satisfying $d(u_1,\alpha_i)\neq d(u_2,\alpha_i)$, i.e., the representations $r(u_1| V_{n-1})$ and $r(u_2|V_{n-1})$ are different. Hence, $V_{n-1}$ is a resolving set. The set $V_{n-1}$ is minimal resolving set because no vertex in $V_{n-1}\setminus\{w\}$ can resolve the vertices $w$ and $w+\alpha_n.$
\end{proof}
\begin{Theorem}\label{exchange}
For $q=2$ and $n\geq 3,$ the exchange property does not hold in $\Gamma(\mathbb{V})$.
\end{Theorem}
\begin{proof}
The size of the set $V_{n-1}$ mention in Lemma~\ref{mresolving} is $2^{n-1}-1$. Since $2^{n-1}-1>n=dim_{\Gamma(\mathbb{V})}$ for all $n\geq 4$ so, by Lemma~\ref{sufcon}, the exchange property does not hold in $\Gamma(\mathbb{V})$ for $n>4$. For $n=3,$ the following resolving set is minimal
$$\{\alpha_1,\alpha_1+\alpha_2,\alpha_2+\alpha_3, \alpha_1+\alpha_2+\alpha_3 \}$$ with cardinality $4>3=dim_{\Gamma(\mathbb{V})}$ and the proof is complete by Lemma~\ref{sufcon}.
\end{proof}
Note that, for $q=2$ and $n=2$, the exchange property holds in $\Gamma(\mathbb{V})$ as there are only two minimal resolving sets $\{\alpha_1\}$ and $\{\alpha_2\}.$
\begin{Theorem}
For $q\geq 3$, the exchange property holds in the graph $\Gamma(\mathbb{V})$.
\end{Theorem}
\begin{proof}
For $q\geq 3$, there is no class $C_v$ of cardinality $1$ for all $v\in V$. By Lemma~\ref{equi}, there is no singleton twin in the graph $\Gamma(\mathbb{V})$. Let $W_{1}\neq W_{2}$ be two minimal resolving sets. Let $u_{1}\in W_1$. Now, if $u_{1}\in W_2,$ then $(W_1\setminus \{u_{1}\})\cup \{u_{1}\}$ is, obviously, a minimal resolving set. Suppose $u_{1}\notin W_2$. Now, since $\Gamma(\mathbb{V})$ is always connected and has no singleton twin. Therefore, by lemma~\ref{verclass}, there exists $ u_{2}\equiv u_1$ such that $ u_{2}\in W_2$. Consequently, again by lemma~\ref{verclass}, $(W_1\setminus \{u_{1}\})\cup \{u_{2}\}$ is a resolving set. Since $d(u_1,w)=d(u_2,w)$ for $u_1\equiv u_2$ and for all $w\in V$. Therefore, two vertices are resolved by the vertex $u_1$ if and only if they are  resolved by the vertex $u_2$. Therefore, $(W_1\setminus \{u_{1}\})\cup \{u_{2}\}$ is  minimal.
\end{proof}
\section{Intersection Graph of $\Gamma(\mathbb{V})$ for $q=2$}
Let $\mathbb{M}$ be the collection of sets. The intersection graph of
$\mathbb{M}$ is the graph $\Gamma(\mathbb{M})=(V,E)$ whose vertices are
the sets in $\mathbb{M}$ and two vertices are connected by an edge
if the corresponding sets intersect. Symbolically\\
$V= \mathbb{M} $ and $ E = \{(A,B)
\subseteq \mathbb{M}: A \cap B \neq \emptyset\}$ (for details see, \cite{Mckee}). The following theorem reflects the importance of intersection graphs which states that every graph can be represented as an intersection graph.
\begin{Theorem}[\cite{Mar}]
Let $\Gamma$ be an arbitrary graph. Then, there is a set  $Y$ and a family $\mathbb{M}$ of subsets $Y_1,Y_2,\ldots,$ of $Y$ which can be put into one-one
correspondence with the vertices of $\Gamma$ in
such a way that the set of  edges is $  E =\{(A,B)
\subseteq \mathbb{M}: A \cap B \neq \emptyset\}$.
\end{Theorem}

The case, where the field $\mathbb{F}$ has two elements, i.e., $q=2$. Then, the graph $\Gamma(\mathbb{V})$
of a vector space $\mathbb{V}$ can be represented by the intersection graph $\Gamma(\mathbb{M})$ for $\mathbb{M}=P(Y)\setminus\emptyset$, where
$Y= \{1,2,3,....n\}$ and $P(Y)$ the power set of $Y$. This
representation can be established by identifying the vector
$\alpha_{i_1}+\alpha_{i_2}+\cdots+\alpha_{i_k}$ with the subset
$\{i_{1},i_{2},\ldots,i_{k}\}$ of $Y$.

\begin{Corollary}
For $\mathbb{M}=P(Y)\setminus\emptyset$, where $Y= \{1,2,3,....n\}$\\
a) $dim_{\Gamma(\mathbb{M})}=n$ and \\
b) $\Gamma(\mathbb{M})$ holds the exchange property if and only if $n=2.$
\end{Corollary}
\begin{proof}
Part a) of the theorem follows from part b) of Theorem~\ref{Mertic Dimension} and part b) of the theorem follows from Theorem~\ref{exchange}.
\end{proof}
\section{Conclusion}
In this paper, we find the metric dimension of the graph $\Gamma(\mathbb{V})$. We study the relation between a metric basis $\Gamma(\mathbb{V})$-basis and a vector space basis $\mathbb{V}$-basis. We proved that, except the case where the corresponding field has $2$ elements and dimension of $\mathbb{V}$ is $\geq 3$, a $\Gamma(\mathbb{V})$-basis has the exchange property like a $\mathbb{V}$-basis. We also concluded that a $\Gamma(\mathbb{V})$-basis contains a $\mathbb{V}$-basis if the corresponding field has $3$ or more elements.

\end{document}